\documentclass{article}
\usepackage[T2A]{fontenc}
\usepackage[cp1251]{inputenc}
\usepackage{amssymb}
\usepackage{amsxtra}
\usepackage{amsmath}
\usepackage{amsthm}
\usepackage{hyperref}

\usepackage[russian,english]{babel}

\DeclareMathOperator{\Cent}{Cent}
\DeclareMathOperator{\Spec}{Spec}

\DeclareMathOperator{\Hom}{Hom\,}

\DeclareMathOperator{\PGL}{PGL}

\DeclareMathOperator{\lev}{lev}

\DeclareMathOperator{\rank}{rank}
\DeclareMathOperator{\Centr}{C}

\newtheorem{lem}{Lemma}
\newtheorem*{thm*}{Theorem}
\newtheorem{thm}{Theorem}

\newtheorem*{cor*}{Corollary}

\let\l\left
\let\r\right

\newcommand{\st}{\scriptstyle}
\newcommand{\ds}{\displaystyle}

\title{Centralizer of the elementary subgroup of an isotropic reductive group}
\author{E. Kulikova\thanks{The author is supported by RFBR 08-01-00756.},
A. Stavrova\thanks{The author is supported by RFBR 09-01-00878 and RFBR 09-01-90304.}}
\begin{document}
\selectlanguage{english}
\maketitle

\section{Introduction}

Let $R$ be a commutative ring with 1, and let $G$ be an isotropic reductive algebraic group over $R$.
In~\cite{PS} Victor Petrov and the second author introduced a notion of an elementary subgroup $E(R)$
of the group of points $G(R)$.

More precisely, assume that $G$ is isotropic in the following strong sense:
it possesses a parabolic subgroup that intersects properly any semisimple normal subgroup of $G$.
Such a parabolic subgroup $P$ is called {\it strictly proper}. Denote
by $E_P(R)$ the subgroup of $G(R)$ generated by the $R$-points of the unipotent
radicals of $P$ and of an opposite parabolic subgroup $P^-$.
The main theorem of~\cite{PS} states that $E_P(R)$
does not depend on the choice of $P$, as soon as
for any maximal ideal $M$ of $R$ all irreducible components of the relative
root system of $G_{R_M}$ (see~\cite[Exp. XXVI, \S 7]{SGA} for the definition) are of rank $\ge 2$.
Under this assumption, we call $E_P(R)$ the {\it elementary
subgroup} of $G(R)$ and denote it simply by $E(R)$.
In particular, $E(R)$ is normal in $G(R)$.
This definition of $E(R)$ generalizes the well-known definition of an elementary subgroup of a Chevalley group (or, more
generally, of a split reductive group), as well as several other definitions of an elementary subgroup
of isotropic classical groups and simple groups over fields.
The group $E(R)$ is also perfect
under natural assumptions on $R$~\cite{LS}. Here we continue this theme by proving that the centralizer of
$E(R)$ in $G(R)$ coincides with the group of $R$-points of the group scheme center
$\Cent(G)$ (see~\cite[Exp. I 2.3]{SGA} for the definition). Consequently, both these subgroups
also coincide with the abstract group center of $G(R)$. Our result extends the
respective theorem of E. Abe and J. Hurly for Chevalley groups~\cite{AbeHur}; see also~\cite[Lemma 2]{St}
for a slighly more general statement.

\begin{thm}\label{th:main}
Let $G$ be an isotropic reductive algebraic group over a commutative ring $R$ having a strictly proper
parabolic subgroup $P$.
Assume that for any maximal ideal $M$ of $R$ all irreducible components of the relative
root system of $G_{R_M}$ are of rank $\ge 2$.
Then $\Centr_{G(R)}(E(R))=\Cent(G)(R)=\Centr(G(R))$.
\end{thm}

Observe that the condition of the theorem ensures that
the elementary subgroup $E(R)$ of $G(R)$ is correctly defined.
We refer to~\cite{LS} for its definition and basic properties, as well as for the preliminaries on relative root
subschemes.

{\bf Remark.} One may ask if the statement holds for $E_P(R)$ instead of $E(R)$, if we do not assume that the local relative
rank is at least 2. This seems to hold always except for several natural exceptions, similar to the exception
for $\PGL_2$ described in~\cite{AbeHur}. We plan to address this case in the near future.

\section{Preliminary lemmas}

We refer to~\cite{LS} and~\cite{PS} for the preliminaries and notation.

We include the following obvious lemma for the sake of completeness.

\begin{lem}\label{lem:closedpt}
Let $X=\Spec A$ be an affine scheme over $Y=\Spec R$, and let $Z$ be a closed subscheme of $X$. Take $g\in X(R)$.
Then $g\in Z(R)$ if and only if $g\in Z(R_M)$ for any maximal ideal $M$ of $R$.
\end{lem}
\begin{proof}
For any $R$-module $V$, the natural map $V\to \prod V\otimes R_M$, where the product runs over all maximal
ideals $M$ of $R$, is injective (e.g.~\cite[p. 104, Lemma]{Wat}). Since $g\in Z(R)$ is equivalent
to an inclusion between the respective ideals of $A$ which are $R$-modules, the Lemma holds.
\end{proof}

\begin{lem}\label{lem:R_P}
Let $R$ be any commutative ring, $G$ an isotropic reductive group over $R$, $P$ a strictly proper parabolic subgroup of $G$.
Take any maximal ideal $M$ of $R$ and any strictly proper parabolic subgroup $P'$ of $G_{R_M}$ contained in $P_{R_M}$.
Then for any $A\in\Phi_{P'}$
there is a system of generators ${e_A}_i$, $1\le i\le n_A$, of the $R_M$-module $V_A$ such that
for all $g$ in the image of $\Cent_{G(R)}(E_P(R))$ in $G(R_M)$, one has $[g,X_A({e_A}_i)]=1$,
$1\le i\le n_A$.
\end{lem}
\begin{proof}
We assume from the very beginning that we have passed to a member of the disjoint union
$$
\Spec(R)=\coprod\limits_{i=1}^m\Spec(R_i),
$$
so that the parabolic subgroup $P$ is also provided with a relative root system $\Phi_P$ and corresponding relative
root subschemes. Since for any $B\in\Phi_P$ elements of $V_B$ generate $V_B\otimes_R R_M$ as an $R_m$-module,
the claim of the lemma holds if $P'=P_{R_M}$.

By~\cite[Lemma 12]{PS}, for any two strictly proper parabolic subgroups
$Q\le Q'$ of a reductive group scheme, one can find such $k>0$ depending only on  $\rank\Phi_Q$, that for any relative root
$A\in\Phi_Q$ and any $v\in V_A$ there exist relative roots
$B_i,C_{ij}\in\Phi_{Q'}$, elements $v_i\in V_{B_i}$,  $u_{ij}\in V_{C_{ij}}$, and integers
$k_i,n_i,l_{ij}>0$ {\rm ($1\le i\le m$, $1\le j\le m_j$)},
which satisfy the equality
$$
X_A(\xi\eta^k v)=\prod\limits_{i=1}^m X_{B_i}(\xi^{k_i}\eta^{n_i} v_{i})^{\ds\mbox{$\prod\limits_{j=1}^{m_i}$}
X_{C_{ij}}(\eta^{l_{ij}} u_{ij})},
$$
where $\xi,\eta$ are free variables. Taking $Q=P'$, $Q'=P_{R_M}$, $\xi=1$,
for any element $v_i$ of a generating system
of the $R_m$-module $V_A$ we get a decomposition
$$
X_A(\eta^k v)=\prod\limits_{i=1}^m X_{B_i}(\eta^{n_i} v_i),
$$
for some $B_i\in\Phi_P$ and $v_i\in V_{B_i}\otimes R_M$, $n_i>0$. Clearly, for any $v_i$ there is an element
$s_i\in R\setminus M$ such that $s_i v_i$ belongs to $V_{B_i}$ (strictly speaking, to the
image of $V_{B_i}$ in $V_{B_i}\otimes R_M$ under the localisation homomorphism; here and below we allow ourselves
this freedom of speech). Set $\eta=s_1\ldots s_m$. Then $X_A(\eta^kv)\in E_P(R)$, and hence $[g,X_A(\eta^kv)]=1$
for any $g\in \Cent_{G(R)}(E_P(R))$.
Thus, multiplying the elements of a generating system of $V_A$ by certain invertible elements of $R_M$, we
obtain a new generating system of $V_A$, which is centralized by $\Cent_{G(R)}(E_P(R))$.
\end{proof}

\begin{lem}\label{lem:omega}
Let $R$ be a local ring (in particular, $R$ can be a field) with the maximal ideal $M$, and let $G$
be a split reductive group over $R$. Let
$P$ be a parabolic subroup of $G$ such that $\rank \Phi_P\ge 2$. Assume that $g\in G(R)$ is such that
for any $A\in\Phi_P$ there is a system of generators ${e_A}_i$, $1\le i\le n_A$, of $V_A$ such that $[g,X_A({e_A}_i)]=1$ for all $i$.
Then $g\in U_P(M)L(R)U_{P^-}(M)$, where $U_{P^\pm}(M)=\l<X_A(MV_A),\ A\in\Phi^\pm\r>$.
\end{lem}
\begin{proof}
First let $R$ be a field. We need to show that $g\in L(R)$. We can assume that $R$ is algebraically closed without
loss of generality.  Let $B^\pm$ be opposite Borel subgroups of $G$ contained in $P^\pm$,
$U^\pm$ be their unipotent radicals, and $T$ their common maximal torus.
Bruhat decomposition implies that $g=uhwv$, where
$u\in U^+(R)$, $h\in T(R)$, $w$ is a representative of the Weyl group,
$v\in U_w^+(R)=\{x\in U^+(R)\ |\ w(x)\in U^-(R)\}$, and this decomposition is unique.
We have $w\in L(R)$ if and only if
$w$ is a product of elementary reflections $w_{\alpha_i}$ for some simple roots $\alpha_i$ belonging to the root system of $L$.

Assume first that $w\not\in L$. Then there is a simple root $\alpha$ not belonging to the root system of $L$ such that
$w(\alpha)<0$. Consider $A=\pi(\alpha)$. Let $e_A\in V_A$ be a vector from the generating set existing by the
hypothesis of the Lemma such that $x_{\alpha}(\xi)$, $\xi\neq 0$, occurs in the canonic decomposition of $x=X_A(e_A)$
into a product of elementary root unipotents from $U^+$. Since $[g,x]=1$, we have $x(uhwv)=(uhwv)x$.
The rightmost factor in the Bruhat decomposition of $x(uhwv)=(xu)hwv$ equals $v$. However,
since $\alpha$ is a positive root of minimal height, it is clear that the rightmost factor in the Bruhat decomposition
of $(uhv)x$ contains  $x_\alpha(\eta+\xi)$ in its canonic decomposition, if $v$ contains $x_\alpha(\eta)$. Therefore,
this rightmost factor is distinct from $v$, a contradiction.

Therefore, $w\in L(R)$. Then for any $x\in U_P(R)$ we have $wxw^{-1}\in U_P(R)$,
hence by the definition of the Bruhat decomposition $v\in L(R)\cap U^+(R)$. This means that
$g=uhwv\in U^+(R)L(R)=U_P(R)(U^+(R)\cap L(R))L(R)=U_P(R)L(R)=P(R)$. Since symmetric reasoning implies
that $g\in P^-(R)$, we have $g\in P(R)\cap P^-(R)=L(R)$.

Now let $R$ be any local ring. Recall that $\Omega_P=U_PLU_{P^-}\cong U_P\times L\times U_{P^-}$ is a principal open subscheme of $G$
(e.g.~\cite[p. 9]{Ma}). Therefore, if the image of $g\in G(R)$ under the natural homomorphism
$G(R)\to G(R/M)$ is in $\Omega_P(R/M)$, then $g\in \Omega_P(R)$. Since by the above
the image of $g$ is in $L(R/M)$, and $\ker(U_{P^\pm}(R)\to U_{P^\pm}(R/M))=U_{P^\pm}(M)$,
we have $g\in U_P(M)L(R)U_{P^-}(M)$.
\end{proof}

\begin{lem}\label{lem:A-A}
Let $G$ be an isotropic reductive group over a local ring $R$, $M$ the maximal ideal of $R$,
$P$ a parabolic subgroup of $G$, $P^-$ an opposite parabolic subgroup.
For any $u\in U_{P^-}(M)$, $v\in U_P(R)$ there exist $u'\in U_{P^-}(M)$, $v'\in U_P(R)$, and $b\in L(R)$
such that $uv=v'bu'$.
\end{lem}
\begin{proof}
The image of $x=uv$ under $p:G(R)\to G(R/M)$ equals $p(v)$,and thus belongs
to $\Omega_P(R/M)$, where $\Omega_P=U_PLU_{P^-}$. Since $\Omega_P$ is a principal open subscheme of $G$, this implies that
$x\in\Omega_P(R)$, that is, $x=v'bu'$. Since $p(u')=1$, we have $u'\in U_{P^-}(M)$.
\end{proof}


\begin{lem}\label{lem:ABe}
Let $G$ be a reductive group over a commutative ring $R$, $P$ a parabolic subgroup of $G$, $A,B\in\Phi_P$ two
non-proportional relative roots
such that $A+B\in\Phi_P$. Assume that $A-B\not\in\Phi_P$, or $A,B$ belong to the image of a simply laced
 irreducible component of the absolute root system of $G$. Take $0\neq u\in V_B$. Any generating system $e_1,\ldots,e_n$ of the $R$-module
$V_A$ contains an element $e_i$ such that $N_{AB11}(e_i,u)\neq 0$.
\end{lem}
\begin{proof}
Assume that $N_{AB11}(e_i,u)=0$ for all $1\le i\le n$.
Consider an affine fpqc-covering $\coprod \Spec S_\tau\to \Spec R$ that splits $G$. There is a member
$S_\tau=S$ of this covering such that the image of $X_B(u)$ under $G(R)\to G(S)$ is non-trivial.
Write
\begin{equation*}
X_B(u)=\prod_{\pi(\beta)=B}x_{\beta}(a_\beta)\cdot
\prod_{i\ge 2}\prod_{\pi(\beta)=iB}x_\beta(c_\beta),
\end{equation*}
where $\pi:\Phi\to\Phi_P$ is the canonical projection of the absolute root system of $G$ onto the relative one,
$x_\beta$ are root subgroups of the split group $G_S$,
and $a_\beta\in S$. Since $X_B(u)\neq 0$, the definition of $X_B$ implies that there exists $a_\beta\neq 0$.
Let $\beta_0\in\pi^{-1}(B)$ be the root of minimal height with this property. By~\cite[Lemma 4]{PS} there exists
a root $\alpha\in\pi^{-1}(A)$ such that $\alpha+\beta_0\in\Phi$. Let $v\in V_A\otimes_R S$ be such that
$X_A(v)=x_\alpha(1)\prod\limits_{i\ge 2}\prod\limits_{\pi(\gamma)=iA}x_\gamma(d_\gamma)$, for some $d_\gamma\in S$.
Then the (usual) Chevalley commutator formula implies that $[X_A(v),X_B(u)]$ contains in its decomposition a
factor $x_{\alpha+\beta}(\lambda a_{\beta_0})$, where $\lambda\in\{\pm 1,\pm 2,\pm 3\}$. However, since either $\alpha,\beta$
belong to a simply laced irreducible component of $\Phi$, or $A-B\not\in\Phi_P$,
we have $\lambda=\pm 1$. Then $N_{AB11}(v,u)\neq 0$, a contradiction.
\end{proof}

Recall~\cite{PS} that any relative root $A\in\Phi_{J,\Gamma}$ can be represented as a (unique)
linear combination of simple relative roots.
The \emph{level} $\lev(A)$ of a relative root $A$ is the sum of coefficients in this decomposition.

\begin{lem}\label{lem:levi}
Let $R$ be a local ring with the maximal ideal $M$, and let $G$
be a reductive group over $R$. Let
$P$ be a parabolic subgroup of $G$ such that $\rank \Phi_P\ge 2$, and the type of $P$ occurs as the type
of a minimal parabolic subgroup of some reductive group over a local ring (not necessarily over $R$).
Assume that $g\in G(R)$ is such that
for any $A\in\Phi_P$ there is a system of generators ${e_A}_i$, $1\le i\le n_A$, of $V_A$ such that $[g,X_A({e_A}_i)]=1$
for all $i$. If
$g\in U_P(M)L(R)U_{P^-}(M)$, then $g\in L(R)$.
\end{lem}
\begin{proof}
Write $g=xhy$, where $x\in U_P(M)$, $h\in L(R)$, $y\in U_{P^-}(M)$. We have
$\prod_{A\in\Phi_P^+} X_A(u_A)$, $y=\prod_{A\in\Phi_P^-} X_A(u_A)$, where the product is taken in any fixed order.


Let $A\in\Phi_P$ be such that $u_A\neq 0$, and $|\lev(A)|$ is minimal among the levels of relative roots with this property.
We are going to deduce a contradiction, thus showing that $A$ cannot occur in the decomposition of $g$.

Assume that $A\in\Phi_P^+$; the other case is treated symmetrically. Since the type of $P$ coincides with the type
of a minimal parabolic subgroup, $\Phi_P$ is isomorphic to a root system
as a set with two partially defined operations---addition and multiplication by integers. Then the standard properties
of a root system imply that one can find a simple root or a minus simple root $B\in\Phi_P$, non-proportional to $A$, such that
$A+B\in\Phi_P$.
Moreover, if the irreducible component of $\Phi_P$ containing $A$ is not of type $G_2$, we can, and we will, choose
$B$ so that
$A-B\not\in\Phi_P$. If it is of type $G_2$, this may be impossible; then
we stipulate that we take $B$ positive. The classification of Tits indices over local rings~\cite{PS-tind} also implies
that in this case the respective irreducible component of the absolute root system of $G$ is either
simply laced or itself of type $G_2$. Assume for now that the latter does not take place;
we will treat this exceptional case in the very end of this proof. Then by Lemma~\ref{lem:ABe} one can find
an element $e$ of a generating system of $V_B$ centralized by $g$ such that $N_{AB11}(u_A,e)\neq 0$.

We have
$1=[X_B(e),g]=[X_B(e),x](x[X_B(e),hy]x^{-1})$. This is equivalent to
\begin{equation}\label{eq:1=}
1=(x^{-1}[X_B(e),x]x)[X_B(e),hy]=[x^{-1},X_B(e)][X_B(e),hy].
\end{equation}
By~\cite[Th. 2]{PS} we can write
$$
x^{-1}=X_A(-u_A)\prod_{
\begin{array}{c}
\st C\in\Phi_P^+,\ C\neq A,\\
\st \lev(C)\ge\lev(A)
\end{array}} X_C(v_C)=X_A(-u_A)\cdot x_1,
$$
and thus
\begin{equation}\label{eq:firstfactor}
\begin{array}{l}
[x^{-1},X_B(e)]=[X_A(-u_A)x_1,X_B(e)]\\
=[X_A(-u_A),[x_1,X_B(e)]]\cdot [x_1,X_B(e)]\cdot [X_A(-u_A),X_B(e)].
\end{array}
\end{equation}

{\bf Case 1: $B$ is positive, that is, $B$ is a simple root.}
We study the factor $[X_B(e),hy]$ of~\eqref{eq:1=}. Write $[X_B(e),hy]=X_B(e)h(yX_B(e)^{-1}y^{-1})h^{-1}$, and
$$
y=\prod_{C\in\Phi_P^-,\ C\not{\,\|} B} X_C(v_c)\cdot \prod_{i>0} X_{-iB}(v_{-iB})=y_1y_2.
$$
Using Lemma~\ref{lem:A-A} we obtain
$yX_B(e)^{-1}=y_1(y_2\cdot X_B(e)^{-1})=y_1\cdot\prod\limits_{i>0}X_{iB}(w_{iB})\cdot b\cdot \prod\limits_{i>0}X_{-iB}(w_{iB})$,
where $b\in L(R)$.
Since  relative roots proportional to $B$ does not occur in the decomposition of $y_1$, and $B$ is a simple root,
the generalized Chevalley commutator formula implies that
$y_1\cdot\prod\limits_{i>0}X_{iB}(w_{iB})=\Bigl(\prod\limits_{i>0}X_{iB}(w_{iB})\Bigr) y_3$, where $y_3\in U_{P^-}(R)$.
Hence $yX_B(e)^{-1}\in \Bigl(\prod\limits_{i>0}X_{iB}(w_{iB})\Bigr)P^{-}(R)$, and also
$$
[X_B(e),hy]\in X_B(e)h\Bigl(\prod\limits_{i>0}X_{iB}(w_{iB})\Bigr)h^{-1}P^-(R)=
\Bigl(\prod\limits_{i>0}X_{iB}(z_{iB})\Bigr)P^-(R).
$$

Now we consider the first factor $[x^{-1},X_B(e)]$ of the right side of~\eqref{eq:1=}. The generalized Chevalley commutator
formula, applied to~\eqref{eq:firstfactor}, says that
$$
[x^{-1},X_B(e)]=\prod\limits_{D\in\Phi_P^+} X_D(w_D).
$$
Moreover, $D=A+B$ is a root of minimal height  in the decomposition~\eqref{eq:firstfactor}
satisfying $w_D\neq 0$; in fact, $w_{A+B}=N_{AB11}(-u_A,e)$. Hence, the whole product
$$
[x^{-1},X_B(e)]\cdot [X_B(e),hy]\in X_{A+B}(N_{AB11}(-u_A,e))\cdot\Bigl(\prod\limits_{i>0}X_{iB}(z_{iB})\Bigr)
\cdot \hspace{-30pt}
\prod_{
\begin{array}{c}
\st C\in\Phi_P^+,\\
\st \lev(C)>\lev(A+B)
\end{array}}\hspace{-30pt} X_C(t_C)
\cdot P^-(R)
$$
does not equal $1$, a contradiction.

{\bf Case 2: $B$ is negative, that is $B'=-B$ is a simple root.}
In this case the generalized Chevalley commutator formula immediately implies $[X_B(e),hy]\in P^-(R)$.
We study~\eqref{eq:firstfactor}. Note that the decomposition of $x_1$ does not contain
$X_{B'}(v_{B'})$, and, if $2B'\in\Phi_P$, also does not contain $X_{2B'}(v_{2B'})$. Indeed,
in the first case we would have $\lev(A)=1$, hence $A$ is a simple relative root, hence $A+B=A-B'$ is not a relative root.
In the second case we would have $\lev(A)=2$, and, since $A+B\in\Phi_P$, $A=A'+B'$ for a simple relative root $A'$.
Since in this case we are in the irreducible component of $\Phi_P$ of type $BC_n$, and $B'$ is an extra-short simple root,
we also have $A'+2B'=A-B\in\Phi_P$. But then by our algorithm we would have taken
$(-A')$ instead of $B$, since $A-(-A')=2A'+B'\not\in\Phi_P$.

The above, together with the fact that $B'=-B$ is a simple root, and the generalized Chevalley commutator formula,
implies that $[x_1,X_B(e)]=\prod\limits_{D\in\Phi_P^+} X_D(w_D)$. Moreover, if $w_D\neq 0$, then
$D\neq A+B$, since $A-B$ is not a relative root by our assumptions, and obviously $D$ is not proportional to $B$.
Further, we see that
for any relative root $D$, occuring in the decomposition of $[X_A(-u_A),[x_1,X_B(e)]]$
or $[X_A(-u_A),X_B(e)]$, the coefficient near any simple root $A_0\neq B'$ in the decomposition of $D$ is greater
or equal to that in the decomposition of $A$. Summing up, the only factor
of the form $X_{A-B}(u)$ in the decompositions of the expressions
$[X_A(-u_A),[x_1,X_B(e)]]$, $[x_1,X_B(e)]$, $[X_A(-u_A),X_B(e)]$ is the factor $X_{A-B}(N_{AB11}(-u_A,e))$ in the third one,
and no commutator of the factors can give a new factor of the form $X_{A-B}(u)$ with $u\neq 0$. Hence,
$[x^{-1},X_B(e)]$ contains $X_{A-B}(N_{AB11}(-u_A,e))\neq 1$ in its decomposition, and
$$
[x^{-1},X_B(e)][X_B(e),hy]\in X_{A-B}(N_{AB11}(-u_A,e))\cdot\hspace{-10pt}\prod_{
\begin{array}{c}
\st F\in\Phi_P^+,\\
\st F\neq A-B\end{array}}\hspace{-10pt}X_F(t_F)\cdot P^-(R)
$$
cannot equal $1$, a contradiction.

{\bf Case $G_2$.} We are left with the case when $\Phi_P$ is of type $G_2$, and moreover the relevant component of the
absolute root system of $G$ is also of type $G_2$. Then we can assume without loss of generality that all components
of the absolute root system are of type $G_2$, and consequently $G$ is quasi-split. There exists a canonical \'etale
extension $R'$ of $R$ such that $G$ is a Weil restriction of a split group $G'$ of type $G_2$ over $R'$,
see~\cite[Exp.~XXIV Prop.~5.9]{SGA}. Then $G_{R'}$ is a direct product of $k$ split groups $G_i$ of type $G_2$. To show that
$g\in L(R)$, it is enough to show that the image $g'$ of $g$ in $G(R')$ is in $L(R')$. We know that
$P_{R'}$ is a Borel subgroup of $G_{R'}$, and, since $\Phi_P$ has no multiple roots,
for any $A\in\Phi_P$ we can identify the root subscheme $X_A(V_A\otimes R')$ with the direct product
of $k$ elementary root subgroups $x_\alpha(R')$ of the groups $G_i$. Considering the relevant projections of $g$
and the generating systems of $V_A$, we are reduced to proving the following: if a point $h\in H(S)$
of a split reductive group $H$ of type $G_2$ centralizes $x_\alpha(u_\alpha)$ for some $u_\alpha\in S^\times$,
for any root $\alpha\in\Psi$, where $\Psi$ is the root system of $H$, then $h$ belongs to the
corresponding split maximal torus. By Lemmas~\ref{lem:closedpt} and~\ref{lem:omega} we can also assume that
the ring $S$ is local with the maximal ideal $N$, and $h=\prod_{\alpha\in\Psi^+}x_\alpha(a_\alpha)\cdot h\cdot \prod_{\alpha\in\Psi^-}x_\alpha(a_\alpha)$,
where all $a_\alpha\in N$. Then the proof goes exactly as in~\cite[Prop. 3]{AbeHur}, substituting
the elements $x_\beta(1)$ and $w_\beta(1)$ by $x_\beta(u_\beta)$ and $w_\beta(u_\beta)=x_\beta(u_\beta)x_{-\beta}(-u_\beta^{-1})x_\beta(u_\beta)$.
\end{proof}

\begin{lem}\label{lem:levigood}
Let $G$ be an isotropic reductive algebraic group over a commutative ring $R$, $P$ a parabolic subgroup of $G$,
$L$ a Levi subgroup of $P$.
Assume that $g\in G(R)$ is such that
for any $A\in\Phi_P$ there is a system of generators ${e_A}_i$, $1\le i\le n_A$, of $V_A$ such that $[g,X_A({e_A}_i)]=1$
for all $i$. If $g\in L(R)$, then $[g,E_P(R)]=1$.
\end{lem}
\begin{proof}
We show that $[g,X_A(V_A)]=0$ for any $A\in\Phi_P^+$ by descending induction on the hight of $A$;
the case $A\in\Phi_P^-$ is symmetric.
By~\cite[Th. 2]{PS} for any $g\in L(S)$ and any $A\in\Phi_P$ there exists a set of homogeneous polynomial maps
$\varphi^i_{g,A}: V_A\to V_{iA}$, $i\ge 1$, such that for any $v\in V_A$ one has
$$
gX_A(v)g^{-1}=\prod_{i\ge 1}X_{iA}(\varphi^i_{g,A}(v)).
$$
Since $\varphi^i_{g,A}$ are homogeneous, $[g,X_A(v)]=1$ for $v\in V_A$ implies $[g,X_A(\lambda v)]=1$ for any $\lambda\in R$.
Also by~\cite[Th. 2]{PS}, there exist a set of homogeneous polynomial maps
$q^i_A:V_A\times V_A \to V_{iA}$, $i>1$, such that
$$
X_A(v)X_A(w)=X_A(v+w)\prod_{i>1}X_{iA}(q^i_A(v,w))
$$
for all $v,w\in V_A$. Assume that $[g,X_A(v)]=[g,X_A(w)]=1$. Then
$$
g X_A(v+w) g^{-1}=gX_A(v)X_A(w)g^{-1}\cdot g\Bigl(\prod_{i>1}X_{iA}(q^i_A(v,w))\Bigr)^{-1}g^{-1}
=1,
$$
since by inductive hypothesis $g$ centralizes $X_{iA}(V_{iA})$ for all $i>0$.
\end{proof}

\section{The proof}

\begin{proof}[Proof of Theorem~\ref{th:main}]
Let $g\in G(R)$ centralize $E(R)=E_Q(R)$, where $Q$ a strictly proper parabolic subgroup of $G$. We are going to show
that $g\in\Cent(G)(R)$.
By Lemma~\ref{lem:closedpt} it is enough to show that $g\in\Cent(G)(R_M)$ for any maximal ideal $M$ of $R$.
Fix an ideal $M$, and set $R'=R_M$. Let $P$ be a minimal parabolic subgroup of $G_{R'}$. By
Lemma~\ref{lem:R_P} for any $A\in\Phi_{P}$
there is a system of generators ${e_A}_i$, $1\le i\le n_A$, of the $R'$-module $V_A$ such that
one has $[g,X_A({e_A}_i)]=1$, $1\le i\le n_A$. Note that $\Phi_P$ is a root system
by~\cite[Exp. XXVI, \S 7]{SGA}, and by the assumption of the theorem all irreducible components of $\Phi_P$
are of rank $\ge 2$.

Let $\coprod\Spec S_\tau\to\Spec R'$ be an fpqc-covering such that $G$ splits over each $\Spec S_\tau$.
It is enough to check that $g\in\Cent(G)(S_\tau)$ for every $\tau$ (here we identify $g$ with its image
under $G(R')\to G(S_\tau)$). Fix one $\tau$, and set $S=S_\tau$ for short.
Again by Lemma~\ref{lem:closedpt} it is enough to show that $g\in\Cent(G)(S_N)$ for any maximal ideal $N$ of $S$.

Since a system of generators ${e_A}_i$, $1\le i\le n_A$, of the $R'$-module $V_A$, also
generates $(V_A\otimes_{R'} S)\otimes_S S_N$ as an $S_N$-module,
$g$ satisfies the conditions of Lemmas~\ref{lem:omega} and~\ref{lem:levi} (for the base ring $S_N$);
hence $g\in L(S_N)$, where $L$ is a Levi subgroup of $P$.
By Lemma~\ref{lem:levigood} this implies that $g$ centralizes $E(S_N)$. Since $G_{S_N}$ is split, it has
a Borel subgroup $B$, and $E(S_N)=E_B(S_N)$. Applying Lemmas~\ref{lem:omega} and~\ref{lem:levi} to $B$ instead of $P$,
we get that $g\in T(S_N)$ for a split maximal subtorus $T$ of $G_{S_N}$. Hence $g\in \Hom(\Lambda/\Lambda_r,S_N)\subseteq
\Hom(\Lambda,S_N)=T(S_N)$, where $\Lambda$ is the weight lattice of $G$, and $\Lambda_r$ is the root sublattice.
Therefore, $g\in\Cent(G)(S_N)$.
\end{proof}




\begin{thebibliography}{99}

\bibitem{AbeHur} E. Abe, J. F. Hurley, Centers of Chevalley groups over commutative rings,
{\it Comm. in Algebra} {\bf 16} (1988), 57--74.

\bibitem{SGA} M.~Demazure, A.~Grothendieck, {\it Sch\'emas en groupes}, Lecture Notes in
Mathematics, Vol. 151--153, Springer-Verlag, Berlin-Heidelberg-New York, 1970.



\bibitem{LS} A. Luzgarev, A.Stavrova, {\it Elementary subgroup of an isotropic reductive group is perfect,}
\href{http://arxiv.org/abs/1001.1105}{http://arxiv.org/abs/1001.1105}; to appear in St. Petersburg Mathematical
Journal.

\bibitem{Ma} H.~Matsumoto, {\it Sur les sous-groupes
arithm\'etiques des groupes semi-simples d\'eploy\'es},
Ann. Sci.  de l'\'E.N.S. $4^e$ s\'erie, tome 2, n. 1 (1969), 1--62.

\bibitem{PS} V. Petrov, A. Stavrova, {\it Elementary subgroups of isotropic reductive groups},
St. Petersburg Math. J. {\bf 20} (2009), 625--644.

\bibitem{PS-tind} V.~Petrov, A.~Stavrova, {\it Tits indices over semilocal rings}, to appear in Transformation Groups.


\bibitem{St} A. Stavrova, {\it Normal structure of maximal parabolic subgroups in Chevalley groups over rings},
Algebra Colloq. {\bf 16} (2009), 631--648.

\bibitem{Wat} W. C. Waterhouse, {\it Introduction to affine group schemes}, Springer-Verlag, New York, 1979.
\end{thebibliography}
\end{document}